\documentclass[11pt]{article}

\usepackage[T1]{fontenc}
\usepackage{lmodern}

\usepackage{amsmath, amssymb, amsthm} 
\usepackage[margin=2.5cm]{geometry}
\usepackage{color, graphicx}
\usepackage[caption=false]{subfig}

\usepackage{etoolbox}
\makeatletter
\patchcmd{\@maketitle}{\LARGE \@title}{\LARGE\bfseries\@title}{}{}

\renewcommand{\@seccntformat}[1]{\csname the#1\endcsname.\quad}
\makeatother

\definecolor{darkblue}{rgb}{0,0,.5}

\usepackage{hyperref}
\hypersetup{
	colorlinks=true,		
	linkcolor=darkblue,		
	citecolor=darkblue,		
	urlcolor=darkblue 		
}

\makeatletter
\def\th@plain{%
	\thm@notefont{}
	\itshape 
}
\def\th@definition{%
	\thm@notefont{}
	\normalfont 
}

\renewenvironment{proof}[1][\proofname]{\par
	\normalfont
	\topsep0\p@\@plus3\p@ \trivlist
	\item[\hskip\labelsep\itshape
	#1\@addpunct{.}]\ignorespaces
}{%
	\qed\endtrivlist
}
\makeatother

\newtheorem{theorem}{Theorem}[section]
\newtheorem{lemma}[theorem]{Lemma}
\newtheorem{corollary}[theorem]{Corollary}
\newtheorem{proposition}[theorem]{Proposition}

\theoremstyle{definition}
\newtheorem{definition}[theorem]{Definition}
\theoremstyle{definition}
\newtheorem{example}[theorem]{Example}
\theoremstyle{definition}
\newtheorem{remark}[theorem]{Remark}

\usepackage[shortlabels]{enumitem}

\allowdisplaybreaks

\newcommand{\N}{\ensuremath{\mathbb N}}

\newcommand{\ran}{\ensuremath{\operatorname{ran}}}
\newcommand{\zer}{\ensuremath{\operatorname{zer}}}
\newcommand{\dom}{\ensuremath{\operatorname{dom}}}

\newcommand{\gra}{\ensuremath{\operatorname{gra}}}
\newcommand{\Fix}{\ensuremath{\operatorname{Fix}}}
\newcommand{\Id}{\ensuremath{\operatorname{Id}}}

\def\beq{\begin{equation}}
\def\eeq{\end{equation}}
\def\beqq{\begin{equation*}}
\def\eeqq{\end{equation*}}
\def\baq{\begin{eqnarray}}
\def\eaq{\end{eqnarray}}
\def\baqn{\begin{eqnarray*}}
\def\eaqn{\end{eqnarray*}}

\if
{
\usepackage[pageref]{backref}
\renewcommand*{\backrefalt}[4]{%
\ifcase #1 %
(Not cited)%
\or
(Cited on p.~#2)%
\else
(Cited on pp.~#2)%
\fi
}

}
\fi

\begin{document}

\title{ Monotonicity of Pairs of Operators and  Generalized Inertial Proximal Method}

\author{
Ba Khiet Le\thanks{Analytical and Algebraic Methods in Optimization Research Group, Faculty of Mathematics and Statistics, Ton Duc Thang University, Ho Chi Minh City, Vietnam. E-mail: \texttt{lebakhiet@tdtu.edu.vn}}, 
~
Zakaria Mazgouri\thanks{LSATE Laboratory, 
National School of Applied Sciences, Sidi Mohammed Ben Abdellah University, Fez, Morocco. E-mail: \texttt{zakaria.mazgouri@usmba.ac.ma}}, ~
and~ 
Michel Th\' era\thanks{Mathematics and Computer Science Department, University of Limoges, 123 Avenue Albert Thomas,
87060 Limoges CEDEX, France. E-mail: \texttt{michel.thera@unilim.fr}}
}


\maketitle

\begin{abstract}
\noindent Monotonicity of pairs of operators is an extension of  monotonicity of operators, which plays an important role in solving non-monotone inclusions. One of  challenging problems in this new tool is how to design the associated mappings to obtain the monotone pairs.  In this paper, we solve this problem and propose a Generalized Inertial Proximal Point Algorithm  \((\mathbf{GIPPA})\) using warped resolvents under the  monotonicity of pairs. The weak, strong and linear convergence of the algorithm under some mild assumptions are established. We also provide numerical examples illustrating the implementability and effectiveness of the proposed method.
\end{abstract}


\paragraph{Mathematics Subject Classification (MSC 2020):}
 49J52, 49J53.  
\section{Introduction}
{
The notion of \emph{monotonicity of pairs of operators} was introduced in \cite{acl} to address the inclusion problem
\begin{equation}\label{main}
0 \in F(x),
\end{equation}
where \(F:\mathcal{H} \rightrightarrows \mathcal{H}\) is a (possibly non-monotone) set-valued operator acting on a Hilbert space \(\mathcal{H}\).
The key idea is that, instead of computing the inverse \(F^{-1}\), which is often difficult or impossible, one may solve the inclusion~\eqref{main} by constructing an auxiliary mapping \(v\) such that \(v^{-1}\) is easy to compute and the pair \((F,v)\) is monotone, meaning that \(F\) and \(v\) share compatible monotonicity properties.
When such a mapping \(v\) is available and \(F\) is single-valued, a forward algorithm \cite{acl} can be employed to solve~\eqref{main}. If \(F\) is possibly set-valued, a Generalized Proximal Point Algorithm \((\mathbf{GPPA})\) can be used \cite{LDT}. In this setting, \(v\) need not be invertible; it suffices that the operator \(\gamma F + v\) be invertible for some \(\gamma>0\). Despite these advances, the problem of constructing a suitable mapping \(v\) for a given operator \(F\) remains open and challenging in general.
In this paper, our first contribution is to address this fundamental issue. We begin with the quadratic programming setting, where \(F=A\) is a non-invertible matrix. In this case, we show how to construct a matrix \(B\) such that the pair \((A,B)\) is monotone. This construction can be achieved either through a slight modification of \(A\) (Theorem~\ref{thm:3.6}) or via a general diagonal decomposition (Theorem~\ref{thm:3.3}).
}

{When \(F=f\) is single-valued and nonlinear, its general form naturally motivates viewing \(v\) as a piecewise linear approximation of \(f\) in order to obtain local convergence. For instance, if \(f\) is differentiable and \(x_k\) is a current iterate, one may take \(v\) as the first-order Taylor approximation of \(f\) at \(x_k\), namely
\[
v(x)=f(x_k)+f'(x_k)(x-x_k).
\]
Since the term \(f(x_k)-f'(x_k)x_k\) is constant, it may be neglected, leading to the simplified choice \(v(x)=f'(x_k)x\), with inverse \(v^{-1}(x)=(f'(x_k))^{-1}x\) whenever \(f'(x_k)\) is invertible. In this case, the algorithm proposed in \cite{acl},
\begin{equation}\label{algonew}
x_{k+1}=x_k-hv^{-1}f(x_k), \quad x_0\in\mathcal{H}, \ h>0,
\end{equation}
reduces to
\[
x_{k+1}=x_k-h\,(f'(x_k))^{-1}f(x_k),
\]
which is precisely the classical Newton method.
}

{However, if \(f'(x_k)\) is non-invertible for some \(k\), the inverse \((f'(x_k))^{-1}\) is no longer available. To overcome this difficulty, we instead fix \(v=f'(x^*)\) at a point \(x^*\) where \(f'(x^*)\) is invertible and for which a solution \(\tilde{x}\) exists in a neighborhood of \(x^*\). This leads naturally to a quasi-Newton--type algorithm. Using the framework of monotonicity of pairs of operators, we establish its local linear convergence (Theorem~\ref{thm310}).
These results support the viewpoint advocated by R.~T.~Rockafellar \cite{Rockafellar1970,Rockafellar2019,Rockafellar2023} that so-called \emph{hidden monotonicity} or \emph{hidden convexity} can play a significant role in the analysis and design of algorithms for nonconvex optimization problems.
}

{
Our second contribution is the introduction of a \emph{Generalized Inertial Proximal Point Algorithm}  \((\mathbf{GIPPA})\),  which extends the Generalized Proximal Point Algorithm \((\mathbf{GPPA})\) proposed in \cite{LDT}. Motivated by recent advances on inertial methods for monotone bilevel equilibrium problems \cite{balhag2022weak,balhag2022weak1}, we develop a generalized inertial proximal scheme for solving~\eqref{main}, based on warped resolvents and the framework of monotonicity of pairs of operators.}
{Inertial proximal algorithms, in which each new iterate depends on the two preceding ones, are widely recognized as discrete analogues of second-order dynamical systems. The incorporation of inertial terms often leads to improved convergence behavior of the generated sequence. From a historical perspective, such methods can be traced back to \cite{AA}, where they were introduced for solving~\eqref{main} in the case where \(F\) is a \emph{maximally monotone} operator and the solution set \(F^{-1}(0)\) is assumed to be nonempty.
}

\smallskip
Here is our proposed inertial scheme.  

\smallskip
\hrule
\vspace{2mm}
\noindent {\bf  \(\mathbf{GIPPA}\) (Generalized Inertial Proximal Point Algorithm)}  
\vspace{2mm}
\hrule
\vspace{2mm}
\noindent {\bf Initialization:} Choose sequences \(\{\gamma_n\}\subset (0,+\infty)\) and \(\{\alpha_n\}\subset [0,1]\). Take arbitrary \(x_0, x_1 \in \mathcal{H}\).

\medskip 

\noindent {\bf Iterative steps:}
\begin{itemize}
\item {\bf Step 1.} For every \(n \ge 1\), given the current iterates \(x_{n-1}, x_n \in \mathcal{H}\), set
\[
y_n := x_n + \alpha_n (x_n - x_{n-1}),
\]
and define
\begin{equation*}\label{mainal}
x_{n+1} := J_{\gamma_n F}^v(y_n),
\end{equation*}
where \(J_{\gamma_n F}^v := (\gamma_n F + v)^{-1} \circ v\) denotes the \emph{warped resolvent} of \(\gamma_n F\) with kernel \(v\) \cite{bc}, and  \(v\) is linear such that the pair \((F,v)\) is monotone.

\item {\bf Step 2.} If \(x_{n+1} = y_n\), then stop; in this case, \(x_{n+1}\) is a solution of \eqref{main}. Otherwise, set \(n := n+1\) and return to Step~1.

\end{itemize}

\medskip
\hrule
\vspace{2mm}

\medskip
%
%

{
The warped resolvent introduced in \cite{bc} extends both the classical resolvent and the \(D\)-resolvent \cite{bbc}, and provides a powerful tool for the study of monotone inclusion problems. The monotonicity of pairs of operators guarantees that the warped resolvent is well-defined (see Remark~\ref{wwarp}) and enables the convergence analysis of algorithms designed to solve non-monotone inclusions \cite{LDT}.
}

{The proposed scheme, referred to as the \emph{Generalized Inertial Proximal Point Algorithm} \((\mathbf{GIPPA})\), reduces to the Generalized Proximal Point Algorithm \((\mathbf{GPPA})\) of \cite{LDT} when \(\alpha_n = 0\) for all \(n \geq 0\). Furthermore, when the kernel operator is chosen as the identity mapping, that is, \(v = \Id\), the method simplifies to the classical Proximal Point Algorithm \((\mathbf{PPA})\); see, for example, \cite{Rockafellar}. Under mild and natural assumptions, we establish weak, strong, and linear convergence results for the proposed algorithm \((\mathbf{GIPPA})\).
}

{The remainder of the paper is organized as follows. Section~2 introduces the notation and reviews the necessary background material. Section~3 is devoted to the construction of associated mappings that ensure the monotonicity of pairs of operators. In Section~4, we present a detailed convergence analysis of the algorithm \((\mathbf{GIPPA})\). Section~5 reports two numerical experiments that illustrate the implementability and effectiveness of the proposed method in comparison with \((\mathbf{GPPA})\). Finally, Section~6 concludes the paper with a summary of the main contributions and a discussion of potential directions for future research.}

\section{Mathematical Background and Notations}\label{sec2}

Throughout this work, \( \mathcal{H} \) denotes a real Hilbert space equipped with the inner product \( \langle \cdot, \cdot \rangle \) and the associated norm \( \|\cdot\| \). %
Let \( F : \mathcal{H} \rightrightarrows \mathcal{H} \) be a set-valued operator. The \emph{domain}, \emph{range}, \emph{graph}, \emph{set of zeros}, and \emph{set of fixed points} of \( F \) are defined, respectively, as:  
\begin{align*}
\dom F &= \{x \in \mathcal{H} : F(x) \neq \varnothing\}, \\
\ran F &= \bigcup_{x \in \mathcal{H}} F(x), \\
\gra F &= \{(x, y) : x \in \mathcal{H}, y \in F(x)\}, \\
\zer F &= \{x \in \mathcal{H} : 0 \in F(x)\}, \\
\Fix F &= \{x \in \mathcal{H} : x \in F(x)\}.
\end{align*}  

The \emph{inverse} of \( F \) is defined by \( F^{-1}(y) = \{x \in \mathcal{H} : y \in F(x)\} \). It is easy to see that \( \dom F = \ran F^{-1} \) and \( \ran F = \dom F^{-1} \).  

The set-valued mapping \( F \) is said to be \emph{monotone} if for all \( (x, x^*), (y, y^*) \in \gra F \),
\[
\langle x^* - y^*, x - y \rangle \geq 0,
\]
and \emph{\(\alpha\)-strongly monotone} if \( \alpha \in (0, +\infty) \) and for all \( (x, x^*), (y, y^*) \in \gra F \),
\[
\langle x^* - y^*, x - y \rangle \geq \alpha \|x - y\|^2.
\]  
Next, we recall the definition of \emph{monotonicity of pairs} introduced in \cite{acl} and used in \cite{LDT}, which generalizes classical monotonicity. Given set-valued mappings \( F_1, F_2 : \mathcal{H} \rightrightarrows \mathcal{H} \), we define:  
\begin{itemize}
\item
The pair \( (F_1, F_2) \) is \emph{monotone} if for all \( x, y \in \mathcal{H} \),
\beqq
\langle F_1(x) - F_1(y), F_2(x) - F_2(y) \rangle \geq 0,
\eeqq 
i.e., for all \( x, y \in \mathcal{H} \), \( x_1^* \in F_1(x) \), \( y_1^* \in F_1(y) \), \( x_2^* \in F_2(x) \), \( y_2^* \in F_2(y) \),
$$
\langle x_1^* - y_1^*, x_2^* - y_2^* \rangle \geq 0.
$$
\item
The pair \( (F_1, F_2) \) is \emph{\(\alpha\)-strongly monotone} if \( \alpha \in (0, +\infty) \) and for all \( x, y \in \mathcal{H} \),
\beq\label{strong}
\langle F_1(x) - F_1(y), F_2(x) - F_2(y) \rangle \geq \alpha \|x - y\|^2,
\eeq
i.e., for all \( x, y \in \mathcal{H} \), \( x_1^* \in F_1(x) \), \( y_1^* \in F_1(y) \), \( x_2^* \in F_2(x) \), \( y_2^* \in F_2(y) \),
$$
\langle x_1^* - y_1^*, x_2^* - y_2^* \rangle \geq \alpha \|x - y\|^2.
$$
\end{itemize}  

\begin{remark}
	\begin{enumerate}
		\item It is straightforward to observe that if \( F \) is monotone (respectively, strongly monotone), then the pair 
		\( (F, \Id) \) is also monotone (respectively, strongly monotone).
		\item  The concept of monotonicity for operator pairs naturally arises in problems with block structures, where such formulations often lead to meaningful adaptations, as illustrated in \cite[Examples 2.1 and 2.2]{LDT}.
	\end{enumerate}	
\end{remark}

\begin{definition}\label{wa}\cite{bc}
Let \( F: \mathcal{H} \rightrightarrows \mathcal{H} \) and \( v: \mathcal{H} \to \mathcal{H} \) with \( \dom v = \mathcal{H} \).  
The \emph{warped resolvent} of \( F \) with kernel \( v \) is defined as  
\[
J_F^v = (F + v)^{-1} \circ v,
\]  
provided that {$F + v$ is injective} and \( \ran v \subseteq \ran(F + v) \).
\end{definition}

%

\begin{remark}\label{wwarp}
{
We show that the monotonicity of the pair \((F,v)\), together with the Lipschitz continuity of \(v^{-1}\), ensures that the warped resolvent \(J_{\gamma F}^v\) is well-defined for every \(\gamma>0\).
Assume that \((F,v)\) is monotone and that \(v^{-1}\) is Lipschitz continuous. Then, for all \(x_1,x_2\in\mathcal{H}\),
\[
\bigl\langle \gamma\bigl(F(x_1)-F(x_2)\bigr)+v(x_1)-v(x_2),\, v(x_1)-v(x_2) \bigr\rangle
\ge \|v(x_1)-v(x_2)\|^2.
\]
Since \(v^{-1}\) is Lipschitz continuous, there exists \(\alpha>0\) such that
\[
\|v(x_1)-v(x_2)\| \ge \alpha \|x_1-x_2\|.
\]
As a consequence, the pair \((\gamma F+v,v)\) is strongly monotone, which implies that \((\gamma F+v)^{-1}\) is single-valued and can be everywhere defined; see \cite[Theorem~6]{acl}.  
}{
We emphasize that the warped resolvent \(J_{\gamma F}^v\) remains well-defined as long as \((\gamma F+v)^{-1}\) exists, even when \(v\) is not invertible. In this more general setting, it is precisely the monotonicity of the pair \((F,v)\) that guarantees the convergence of the associated algorithms.
}
\end{remark}

{Strong and weak convergence of sequences are denoted by \(\to\) and \(\rightharpoonup\), respectively. A single-valued mapping \(F:\mathcal{H}\to\mathcal{H}\) is said to be \emph{weakly continuous} if it is continuous with respect to the weak topology of \(\mathcal{H}\); that is, for every sequence \((x_n)_{n\in\mathbb{N}}\subset\mathcal{H}\) such that \(x_n\rightharpoonup x\), one has \(F(x_n)\rightharpoonup F(x)\).
}

{A set-valued mapping \(F:\mathcal{H}\rightrightarrows\mathcal{H}\) is said to have a \emph{strongly--weakly closed graph} if, for every sequence \((x_n,y_n)_{n\in\mathbb{N}}\subset\gra F\) with \(x_n\rightharpoonup x\) and \(y_n\to y\), it holds that \((x,y)\in\gra F\), that is, \(y\in F(x)\). Moreover, we say that \(F\) has a \emph{closed graph at a point} \(\bar x\in\dom F\) if, for every sequence \((x_n,y_n)_{n\in\mathbb{N}}\subset\gra F\) satisfying \(x_n\to \bar x\) and \(y_n\to y\), one has \((\bar x,y)\in\gra F\).
}

On the other hand, we recall the notion of \emph{\(R\)-continuity}, introduced in \cite{L1,LT}. Let \(F:\mathcal{H}\rightrightarrows\mathcal{H}\) be a set-valued mapping and let \(\bar x\in\dom F\).

{
\begin{definition}
The set-valued mapping \(F:\mathcal{H}\rightrightarrows\mathcal{H}\) is said to be \emph{\(R\)-continuous} at \(\bar x\) if there exist \(\sigma>0\) and a nondecreasing function \(\rho:\mathbb{R}^+\to\mathbb{R}^+\) satisfying
\[
\lim_{r\to 0^+}\rho(r)=\rho(0)=0
\]
such that
\begin{equation*}
F(x)\subset F(\bar x)+\rho\bigl(\|x-\bar x\|\bigr)\mathbb{B},
\qquad \forall\,x\in\mathbb{B}(\bar x,\sigma).
\end{equation*}
{Here $\mathbb{B}$ denotes the closed unit ball, while $\mathbb{B}(\bar x,\sigma)$ denotes the closed ball with center $\bar x$ and radius $\sigma> 0$}.
The function \(\rho\) is called a \emph{continuity modulus} of \(F\) at \(\bar x\), and \(\sigma\) is referred to as the \emph{radius}.
\end{definition}
}

We note that \(R\)-continuity is satisfied by a broad class of operators. In particular, if \(F\) has a closed graph and is locally compact at \(\bar x\), then \(F\) is \(R\)-continuous at \(\bar x\); see \cite[Theorem~3.2]{LT}.

We also require the following technical lemma.

{
\begin{lemma}\label{lem5-a}
For all \(x,y\in\mathcal{H}\) and all \(\beta\in\mathbb{R}\), the following identity holds:
\[
\|\beta x+(1-\beta)y\|^{2}
= \beta\|x\|^{2} + (1-\beta)\|y\|^{2}
- \beta(1-\beta)\|x-y\|^{2}.
\]
\end{lemma}
}

\begin{lemma}\label{lem-sum}
	Let   $\{b_k\}_{k\geq0}, \{w_k\}_{k\geq0}$ and $\{\theta_k\}_{k\geq0}$   be sequences of nonnegative numbers, and there exists a real number $\theta$ with $0 \leq \theta_k \leq \theta \le 1,$   such that for all $k\geq 0$, 
	\begin{equation*}
	b_{k+1}\leq \theta_k b_k+w_k
	\end{equation*} 
	with 
	 \;$\sum_{k=0}^{+\infty} w_k< +\infty$, then $\displaystyle\lim_{k\rightarrow +\infty}b_k$ exists.  Further, if $\theta<1$ then \;$\sum_{k=0}^{+\infty} b_k< +\infty$.
\end{lemma}

\begin{proof}
	Since  $b_k\geq 0$  for all $k$ and  $\sum_{k=0}^{+\infty} w_k< +\infty,$ then the sequence $\{b_{k+1}-\sum_{i=0 }^{k} w_i\}$ is bounded  from below.
	We also have 
	$$b_{k+1}- \sum_{i=0 }^{k}  w_i\leq \theta_k b_{k}- \sum_{i=0 }^{k-1}  w_i \leq  b_{k}- \sum_{i=0 }^{k-1}  w_i,$$			
	which implies that the sequence $\{b_{k+1}- \sum_{i=0 }^{k}  w_i\}$  is nonincreasing, hence convergent. It follows that  $\displaystyle \lim_{k\rightarrow +\infty}b_k$ exists.\\
	Observe that $$(1-\theta_k)b_k\leq b_k-b_{k+1}+w_k.$$	
	Summing up from  $k=0$ to $n$, we get
	\begin{align*}
	(1-\theta)\sum_{k=0}^{n}b_k&\leq \sum_{k=0}^{n}(b_k-b_{k+1})+\sum_{k=0}^{n}w_k\\
	&=b_0-b_{n+1}+\sum_{k=0}^{n}w_k \\
	&\leq b_0+\sum_{k=0}^{n}w_k.
	\end{align*}
	Since $\sum_{k=0}^{+\infty} w_k< +\infty$, if $\theta<1$, we conclude that $\sum_{k=0}^{+\infty} b_k< +\infty$.  
\end{proof} 
\begin{lemma}\label{lem-sum1}
	Let $0\leq \alpha <1$, and let $\{\alpha_k\}, $ $\{a_k\},$ $\{\Delta_k\}$ and $\{w_k\}$ be sequences of nonnegative numbers such that $\{\alpha_k\} $ is nondecreasing, $\left\{\alpha_k\right\}\subseteq[0, \alpha]$ and  for all $k\geq 1$,  
	\begin{equation}\label{de}
	a_{k+1}\leq (\alpha_k+1)a_{k}- \alpha_{k-1} a_{k-1}-\Delta_k+w_k.
	\end{equation} 
	If \;$\sum_{k=0}^{+\infty} w_k< +\infty$, then
	\begin{description}
		\item[$i)$]	 $a_k$ is bounded;
		\item[$ii)$]  $\sum_{k=0}^{+\infty}\Delta_k< +\infty$;
		\item[$iii)$]  $\displaystyle \lim_{k\rightarrow +\infty}a_k$ exists.
	\end{description}
\end{lemma}
\begin{proof}
	 $i)$
	By summing up the inequality \eqref{de} from $k=1$ to $n$, we obtain   
	$$a_{n+1}-a_{1}\leq \left(\alpha_n a_{n}-\alpha_0 a_{0}\right)-\sum_{k=1}^{n}\Delta_k+\sum_{k=1}^{n}w_k,$$
	which implies that
	$a_{n+1}\leq \alpha a_{n}+ a_{1}+ \sum_{k=1}^{n}w_k.$
	Note that  $\sum_{k=0}^{+\infty} w_k< +\infty$ we conclude that 
	$\{a_n\}$ is bounded by some $C>0$. 
	
	 $ii)$ We have 
	$$ \sum_{k=1}^{n}\Delta_k \leq a_{1} +\alpha a_{n}+\sum_{k=1}^{n}w_k.$$
	Since $\{a_n\}$ is bounded and again $\sum_{k=0}^{+\infty} w_k< +\infty$, then  $\sum_{k=1}^{+\infty}\Delta_k<+\infty.$
	
	 $iii)$ From (\ref{de}), we imply that 
	 \baqn
	 a_{k+1}-a_k&\le& \alpha_ka_{k}- \alpha_{k-1} a_{k-1}+w_k= (\alpha_k-\alpha_{k-1})a_{k}+\alpha_{k-1}(a_{k}-a_{k-1})+w_k\\
	 &\le&\alpha[a_k-a_{k-1}]_+ + (\alpha_k-\alpha_{k-1}) C+w_k.
	 \eaqn
Thus 
	\begin{equation*}
	[a_{k+1}-a_k]_+\leq\alpha[a_k-a_{k-1}]_+ + (\alpha_k-\alpha_{k-1}) C+w_k.
	\end{equation*}
	Using the fact that $\displaystyle \sum_{k=1}^{+\infty}( (\alpha_k-\alpha_{k-1}) C+w_k)\leq \alpha C+\sum_{k=1}^{+\infty}w_k<+\infty$ and applying Lemma \ref{lem-sum} with $b_k=[a_k-a_{k-1}]_+,$  we obtain    
	$$\sum_{k=1}^{+\infty}[a_k-a_{k-1}]_+<+\infty.$$ 
	For any \( n \geq 1 \), we can write
		\[
		a_n = a_1 + \sum_{k=2}^{n} (a_k - a_{k-1}).
		\]
		Split each term into its positive and negative parts
		\[
		a_k - a_{k-1} =  {[a_k - a_{k-1}]_+ - [a_k-a_{k-1} ]_-.}
		\]
		Hence,
		\[
		a_n = a_1 + { \sum_{k=2}^{n} [a_k - a_{k-1}]_+ - \sum_{k=2}^{n}  [a_k-a_{k-1} ]_-}.
		\]
		
		\noindent Now, since $\{a_n\}$ is nonnegative {we imply that
			$$\sum_{n=1}^{+\infty}[a_n-a_{n-1}]_-<+\infty.$$ 
			}
		 {Thus the series $\sum_{n=1}^{+\infty}[a_n-a_{n-1}]_+$ and $\sum_{n=1}^{+\infty}[a_n-a_{n-1}]_-$ are convergent. Therefore  $\{a_n\}$ is convergent. }
\end{proof}

We conclude this section by recalling the classical weak convergence result due to Opial.  

\begin{lemma}[Opial's Lemma {\cite{Opial}}]
	\label{l:Opial}
	Let \( S \) be a nonempty subset of \( \mathcal{H} \), and let \( (x_n)_{n \in \mathbb{N}} \) be a sequence in \( \mathcal{H} \). Suppose that:  
	\begin{enumerate}
		\item
		For every \( x^* \in S \), \(\displaystyle \lim_{n \to \infty} \|x_n - x^*\| \) exists.  
		\item
		Every sequential weak cluster point of \( (x_n)_{n \in \mathbb{N}} \) belongs to \( S \).  
	\end{enumerate}
	Then the sequence \( (x_n)_{n \in \mathbb{N}} \) converges weakly to some point \( x_\infty \in S \).  
\end{lemma}

\section{Constructing Associated Mappings to Obtain Monotone Pairs}
\subsection{Global Monotonicity}

{We first consider a simple yet fundamental setting arising in quadratic programming, where \(F=A\) is a non-invertible matrix. The central question is how to construct a matrix \(B\) such that the pair \((A,B)\) is monotone, or equivalently, such that the operator \(A^{\top}B\) is monotone. In addition, it is desirable that \(B\) be invertible, or at least that \(\gamma A+B\) be invertible for some \(\gamma>0\).
}
We begin by presenting a construction of \(B\) based on a slight modification of the matrix \(A\).

{
\begin{theorem}\label{thm:3.6}
Let \(A\) be a non-invertible matrix. Suppose that \(B=A+A_1\), where \(A_1\) is chosen so that \(A_1^{\top}A\) is monotone. Then the pair \((A,B)\) is monotone.
\end{theorem}
}
{
\begin{proof}
For any \(x\), we compute
\[
\langle Ax, Bx\rangle
= \langle Ax, (A+A_1)x\rangle
= \|Ax\|^2 + \langle A_1^{\top}Ax, x\rangle
\ge \|Ax\|^2
\ge 0,
\]
where the inequality follows from the monotonicity of \(A_1^{\top}A\). This proves that the pair \((A,B)\) is monotone.
\end{proof}
}
\begin{remark}
{ For example if 
$$A =\begin{bmatrix} 1 & 2 & 3\\ 4 & 5 & 6 \\ 7 & 8 & 9\end{bmatrix}, \;{\rm we \;can\; choose}\; 
B_1 =\begin{bmatrix} 1 & 0 & 0\\ 0 & 0 & 0 \\ 0 & 0 & 0\end{bmatrix}, B_2 =A+B_1=\begin{bmatrix} 2 & 2 & 3\\ 4 & 5 & 6 \\ 7 & 8 & 9\end{bmatrix}. $$
Then $(A,B_1)$ and $(A,B_2)$ are monotone where $B_2$ is invertible.
}
\end{remark}

 The second approach is based on a diagonal decomposition.  

\begin{proposition}
\label{prop:3.1}
Suppose that $A$ is a symmetric non-invertible matrix. Then we can decompose 
\[
A = O D O^\top,
\] 
where $O$ is orthogonal and $D$ is diagonal with some zero entries on the main diagonal. Let $D'$ be the diagonal matrix obtained from $D$ by replacing the zero entries with $1$. Let 
\[
B = O D' O^\top.
\] 
Then $B$ is invertible and $(A,B)$ is monotone.
\end{proposition}

\begin{proof}
The existence of such a decomposition is classical. Then we have
\begin{align*}
AB^\top &= O D O^\top (O D' O^\top)^\top \\
        &= O D O^\top O D' O^\top \\
        &= O D D' O^\top,
\end{align*} 
which is monotone.
\end{proof} 
\begin{remark}
\label{rem:3.2}
The choice of $D'$ can be flexible to ensure the monotonicity of $D D'$. We can also easily choose $D'$ which is non-invertible. In addition, the orthogonal decomposition of $A$ can be extended  as follows.
\end{remark}

\begin{theorem}
\label{thm:3.3}
Suppose that 
\[
A = C D E,
\] 
where $C,E$ are invertible matrices and $D$ is diagonal with some zero entries on the main diagonal. Let $D'$ be the diagonal matrix obtained from $D$ by replacing the zero entries with $1$. Let 
    \[
    B = (C^{-1})^\top D' E.
      \] 
Then $B$ is invertible and  $(A,B)$ is monotone. 
\end{theorem}

\begin{proof}
It is obvious that $B$ is invertible since $C, D', E$ are invertible. We have
\begin{align*}
A^\top B &= (C D E)^\top (C^\top)^{-1} D' E \\
         &= E^\top D C^\top (C^\top)^{-1} D' E \\
         &= E^\top D D' E,
\end{align*} 
which is monotone. 
\end{proof}

\subsection{Local Strong Monotonicity and {Quasi-Newton's Algorithm}}

In this section, we establish local strong monotonicity of pairs of operators as a key ingredient for proving local linear convergence.
Assume that the equation \( f(x) = 0 \) admits a solution \( \tilde{x} \) in a neighborhood of a fixed point \( x^* \), and that \( f \) is \( C^1 \) in this neighborhood with \( f'(x^*) \) invertible.
The Taylor expansion of \( f \) around \( x^* \) yields
\[
f(x) = f(x^*) + f'(x^*)(x - x^*) + o(\|x - x^*\|).
\]
We choose \( v \) as the linear part of \( f \) at \( x^* \), namely
\[
v(x) := f'(x^*)x,
\]
since the constant term \( f(x^*) - f'(x^*)x^* \) plays no role in monotonicity arguments.
\begin{theorem}\label{locs}
Let \( v(x) = f'(x^*)x \). Then the pair \( (f,v) \) is locally strongly monotone, i.e., there exist constants \( \varepsilon > 0 \) and \( \alpha > 0 \) such that for all
\( x, y \in B(x^*,\varepsilon) \),
\[
\langle f(x) - f(y),\, v(x) - v(y) \rangle
\ge \alpha \|x - y\|^2.
\]
\end{theorem}
\begin{proof}
Fix \( x, y \in B(x^*,\varepsilon) \) {where $\varepsilon>0$ can be chosen later}.
By the Mean Value Theorem, there exists \( \xi \in [x,y] \) such that
\[
f(x) - f(y) = f'(\xi)(x-y).
\]
Therefore,
\begin{align*}
\langle f(x) - f(y),\, v(x) - v(y) \rangle
&= \langle f'(\xi)(x-y),\, f'(x^*)(x-y) \rangle \\
&= \|f'(x^*)(x-y)\|^2
+ \langle (f'(\xi)-f'(x^*))(x-y),\, f'(x^*)(x-y) \rangle.
\end{align*}
Since \( f'(x^*) \) is invertible, there exists \( c > 0 \) such that
\[
\|f'(x^*)(x-y)\| \ge c \|x-y\|.
\]
Moreover, by continuity of \( f' \), for \( \varepsilon \) sufficiently small,
\[
{\| f'(x^*)\| }\|f'(\xi)-f'(x^*)\| \le \frac{c}{2}.
\]
Combining these estimates yields
\[
\langle f(x) - f(y),\, v(x) - v(y) \rangle
\ge \frac{c^2}{2} \|x-y\|^2.
\]
Setting \( \alpha := c^2/2 \) completes the proof.
\end{proof}
Note that Newton's method is a special case of Algorithm~\eqref{algonew}, with
\[
v_k(x) = f'(x_k)x.
\]
However, the method may fail to be well-defined if there exists an iterate \( x_k \) such that
\( f'(x_k) \) is not invertible. To circumvent this difficulty, one may replace
\( f'(x_k) \) with a fixed invertible operator \( f'(x^*) \), where \( x^* \) is a reference point.
Setting
\[
v(x) := f'(x^*)x.
\]

By Theorem~\ref{locs}, there exists \(\varepsilon>0\) such that the pair \((f,v)\) is strongly monotone on \(B(x^*,\varepsilon)\); that is, there exists \(\alpha>0\) such that
\begin{equation}\label{stm}
\langle f(x)-f(y),\, v(x)-v(y) \rangle
\ge \alpha\|x-y\|^2,
\end{equation}
for all \(x,y\in B(x^*,\varepsilon)\).
Consequently, there exists \(\varepsilon_1>0\) such that
\[
W:=v^{-1}\bigl(B(v(x^*),\varepsilon_1)\bigr)\subset B(x^*,\varepsilon).
\]
Note that \(W\) is a neighborhood of \(x^*\).
We consider the quasi-Newton iteration
\begin{equation}\label{quan}
x_{k+1}
= x_k - h\,(f'(x^*))^{-1}f(x_k),
\qquad x_0\in W,\quad h>0.
\end{equation}

\begin{theorem}\label{thm310}
Assume that the equation \(f(x)=0\) admits a solution \(\tilde{x}\) in \(W\).
Then \(\tilde{x}\) is the unique solution in \(W\), and the quasi-Newton iteration~\eqref{quan}
converges linearly to \(\tilde{x}\).
\end{theorem}

\begin{proof}
The uniqueness of \(\tilde{x}\) in \(W\) follows directly from the strong monotonicity of the pair \((f,v)\) on \(B(x^*,\varepsilon)\).
Define
\[
y_k := v(x_k)=f'(x^*)x_k,
\qquad
\tilde{y} := v(\tilde{x})=f'(x^*)\tilde{x}.
\]
Then the iteration~\eqref{quan} can be equivalently written as
\[
y_{k+1}=y_k-hf(x_k).
\]

We first show that \(x_k\in W\) for all \(k\ge0\).
Clearly, \(x_0\in W\).
Assume that \(x_k\in W\).
By~\eqref{stm}, we have
\[
\langle y_k-\tilde{y},\, f(x_k)-f(\tilde{x}) \rangle
\ge \alpha\|x_k-\tilde{x}\|^2.
\]
Let \(L_f\) denote the Lipschitz constant of \(f\) on \(B(x^*,\varepsilon)\).
Choosing \(h=\alpha/L_f^2\), we obtain
\begin{align*}
\|y_{k+1}-\tilde{y}\|^2
&= \|y_k-\tilde{y}-h(f(x_k)-f(\tilde{x}))\|^2 \\
&\le \|y_k-\tilde{y}\|^2
-2h\alpha\|x_k-\tilde{x}\|^2
+h^2L_f^2\|x_k-\tilde{x}\|^2 \\
&= \|y_k-\tilde{y}\|^2
-\frac{\alpha^2}{L_f^2}\|x_k-\tilde{x}\|^2.
\end{align*}
Since \(f'(x^*)\) is invertible, there exists \(L_v>0\) such that
\[
{L_v}\|x_k-\tilde{x}\|\le \|y_k-\tilde{y}\|.
\]
Consequently,
\[
\|y_{k+1}-\tilde{y}\|^2
\le\Bigl(1-\frac{\alpha^2}{L_f^2L_v^2}\Bigr)\|y_k-\tilde{y}\|^2,
\]
which implies that \(x_{k+1}\in W\) and establishes the linear convergence of \((x_k)\) to \(\tilde{x}\).
\end{proof}

\section{Generalized Inertial Proximal Point Algorithm (GIPPA) and  Convergence Analysis}\label{sec4}

In this section, we analyze the convergence of algorithm $(\mathbf{GIPPA})$ using warped resolvents under the monotonicity of pairs of operators, to address the non-monotone inclusion \eqref{main}. We begin by introducing the following assumptions.

\noindent \textbf{Assumption 1}: There exists a linear mapping \( v: \mathcal{H} \to \mathcal{H} \) such that \( (F, v) \) is monotone,  $\gamma_n F + v$ is injective and \( \ran v \subset \ran (\gamma_n F + v) \) where $(\gamma_n)_{n\in \mathbb{N}}$ is a sequence of positive real numbers.

\noindent \textbf{Assumption 1'}: There exists a  linear mapping \( v: \mathcal{H} \to \mathcal{H} \) such that  \( (F, v) \) is \( \beta \)-strongly monotone and \( \ran v \subset \ran (\gamma_n F + v) \) where $(\gamma_n)_{n\in \mathbb{N}}$ is a sequence of positive real numbers.

\medskip
As a preliminary step toward proving the convergence of the algorithm $(\mathbf{GIPPA})$, we establish the following control lemma.
\begin{lemma}\label{lem5}
	 {Suppose that $\zer F\neq \varnothing$ and Assumption 1 holds. Let  $x^*\in \zer F$ and set $a_n:=\|v(x_n)-v(x^*)\|^2$.} Then, for each $n\geq 1$ the following inequality holds:
	\begin{equation}\label{estim1}
	\begin{array}{l}
	a_{n+1}-a_n-{\alpha_n }(a_n-a_{n-1})\\
	\leq(\alpha_n-1)\|v(x_{n+1})-v(x_{n})\|^2 +2\alpha_n\|v(x_{n})-v(x_{n-1})\|^2.
	\end{array}
	\end{equation}
\end{lemma}

\begin{proof}
 By \cite[Proposition 3.1]{LDT}
, $x^*\in \Fix J_{\gamma_n F}^v$, i.e., $(x^*, x^*)\in \gra J_{\gamma_n F}^v$. Since $(y_n,x_{n+1})\in \gra J_{\gamma_n F}^v$, 
we have $$ x_{n+1}\in J_{\gamma_n F}^v(y_n)=(\gamma_n F+v)^{-1} \circ v(y_n).$$
Then, $v(y_n)\in \gamma_n F(x_{n+1})+v(x_{n+1})$, which yields 
\beq\label{mainin}
 v(y_n)-v(x_{n+1})\in \gamma_n F(x_{n+1}). 
 \eeq
Similarly, $(x^*, x^*)\in \gra J_{\gamma_n F}^v$ implies $$ v(x^*)-v(x^*)\in \gamma_n F(x^*). $$
Since the monotonicity of $(F,v)$, we have 
$$ \langle v(y_n) - v(x_{n+1}) - v(x^*) + v(x^*), v(x_{n+1}) - v(x^*) \rangle\ge 0, $$
or equivalently, 
\baqn \nonumber
 \|v(x_{n+1}) - v(x^*)\|^2 &\leq& \langle v(y_n) - v(x^*), v(x_{n+1}) - v(x^*) \rangle\\
&=&\frac{1}{2}(\|v(y_n) -v(x^*)\|^2+ \|v(x_{n+1}) - v(x^*)\|^2-\|v(y_n) -v(x_{n+1})\|^2),
\eaqn
which implies that
\begin{align}\label{eq:vv*}
\|v(x_{n+1}) -v(x^*)\|^2\leq \|v(y_n) -v(x^*)\|^2 -\|v(y_n) -v(x_{n+1})\|^2.
\end{align}
On the other hand, {since $v$ is linear} and by Lemma \ref{lem5-a}, we have for all $n\geq 1$  	
\begin{align}\label{8a}
\|v(y_n)-v(x^*)\|^2&=\|v(x_n+\alpha_n(x_n-x_{n-1}))-v(x^*)\|^2 \nonumber\\
&=\|v(x_n)+\alpha_nv(x_n)-\alpha_nv(x_{n-1})-v(x^*)\|^2\nonumber\\
&=\|(1+\alpha_n)(v(x_n)-v(x^*))-\alpha_n(v(x_{n-1})-v(x^*)\|^2\nonumber\\
&=(1+\alpha_n)\|v(x_n)-v(x^*)\|^2-\alpha_n \|v(x_{n-1})-v(x^*)\|^2\nonumber\\
&\quad+\alpha_n(1+\alpha_n)\|v(x_n)-v(x_{n-1})\|^2.
\end{align}
Also, we have 	
\begin{equation}\label{8b}
\begin{array}{lll}
\|v(x_{n+1})-v(y_n)\|^2& = & \|v(x_{n+1})-v(x_n)-\alpha_n(v(x_n)-v(x_{n-1}))\|^2 \\ 
& = & \|v(x_{n+1})-v(x_n)\|^{2}+\alpha_n^2\|v(x_{n})-v(x_{n-1})\|^{2}\\
&&-2\alpha_n\langle v(x_{n+1})-v(x_n),v(x_{n})-v(x_{n-1})\rangle\\
& \geq &(1-\alpha_n)\|v(x_{n+1})-v(x_n)\|^{2}+(\alpha_n^2-\alpha_n)\|v(x_{n})-v(x_{n-1})\|^{2}.
\end{array} 
\end{equation}
Combining \eqref{8a} and \eqref{8b} with \eqref{eq:vv*}, we get 
\begin{equation}\label{15principal}
\begin{array}{l}
\|v(x_{n+1})-v(x^*)\|^{2}\\
\leq(1+\alpha_n)\|v(x_{n})-v(x^*)\|^{2}-\alpha_n\|v(x_{n-1})-v(x^*)\|^{2}+\alpha_n(1+\alpha_n)\|v(x_n)-v(x_{n-1})\|^{2}\\
\quad-(1-\alpha_n)\|v(x_{n+1})-v(x_{n})\|^2-(\alpha_n^2-\alpha_n)\|v(x_{n})-v(x_{n-1})\|^{2}.
\end{array}
\end{equation}
Finally, by setting $a_n=\|v(x_n)-v(x^*)\|^{2}$ in \eqref{15principal}, we get
\begin{equation*}
\begin{array}{rcl}
a_{n+1}&\leq &(1+\alpha_n)a_n-\alpha_n a_{n-1} +(\alpha_n-1)\|v(x_{n+1})-v(x_{n})\|^2+2\alpha_n\|v(x_n)-v(x_{n-1})\|^2,
\end{array}
\end{equation*}
and then the required inequality \eqref{estim1} follows directly from this last one which ends the proof.
\end{proof}

We may proceed with our analysis under the assumption that $\displaystyle \sum_{n=1}^{+\infty}\|v(x_n)-v(x_{n-1})\|^2<+\infty$; however, this condition depends on the trajectory $\{x_n\}$, which is not known a priori. In the following result, we show that this condition is satisfied provided the sequence of inertial parameters $\lbrace \alpha_n \rbrace$ is suitably controlled.
\begin{corollary}\label{coro-disc}
	 Assume that the sequence $\{\alpha_n\}$ is nondecreasing and $\{\alpha_n\} \subseteq [0,\alpha]$ for some $\alpha\in [0,\frac{1}{3}[,$ then  \\
	 
	{ (i) the sequence $(a_n)$ is bounded};\\
	 
	 (ii) $\displaystyle \sum_{n=1}^{+\infty}\|v(x_n)-v(x_{n-1})\|^2<+\infty.$
\end{corollary}

\begin{proof}
	(i) Return to Lemma \ref{lem5}, by setting $\delta_n=\|v(x_n)-v(x_{n-1})\|^2$, then inequality \eqref{estim1} gives
	\begin{equation}\label{pluie}
	a_{n+1}-a_n-\alpha_n(a_n-a_{n-1})\leq (\alpha_n-1)\delta_{n+1}+2\alpha_n\delta_{n}.
	\end{equation}
	Using the fact that the sequence $\{\alpha_n\}$ is nondecreasing,  one can conclude that inequality \eqref{pluie} implies for $n\geq1$
	\begin{equation*}
	\begin{array}{l}
	a_{n+1}-a_n-\left[\alpha_n a_n- \alpha_{n-1}a_{n-1}\right]+\left[(1-\alpha_n)\delta_{n+1}-(1-\alpha_{n-1})\delta_{n}\right]+\left(1-3\alpha_n\right)\delta_n\leq 0.
	\end{array}
	\end{equation*}
	Let be $N\in \mathbb{N}^*,$ summing up  from $n=1$ to $N$  the above inequality, we get
	\begin{equation}\label{arg}
	\begin{array}{l}
	(a_{N+1}-a_1)- \left[\alpha_N a_N- \alpha_{0}a_{0}\right]+ \left[(1-\alpha_{N})\delta_{N+1}-(1-\alpha_{0})\delta_{1}\right]+\sum_{n=1}^N (1-3\alpha_n)\delta_n
	\leq 0.
	\end{array}
	\end{equation} 
	The fact that $\alpha_n\leq \alpha,$ for $n\geq 0,$  gives 
	\begin{equation}\label{soleil}
	\begin{array}{l}
	(a_{N+1}-a_1)- \left[\alpha a_N- \alpha_{0}a_{0}\right]+ \left[(1-\alpha)\delta_{N+1}-(1-\alpha_{0})\delta_{1}\right]+(1-3\alpha)\sum_{n=1}^N \delta_n
	\leq 0,
	\end{array}
	\end{equation} 
	and then
	\begin{equation}\label{C1}
	(a_{N+1}-\alpha a_N)+(1-\alpha)\delta_{N+1}+(1-3\alpha)\sum_{n=1}^N\delta_n\leq C
	,\end{equation}
	where $C=a_1-\alpha_0 a_0+(1-\alpha_0)\delta_1\in \mathbb{R}.$\\
	Since $\alpha <\frac{1}{3}$ yields $1-3\alpha>0$ and $1-\alpha>0$, then inequality \eqref{C1} implies that  for all $N\geq 1$
	\begin{equation}\label{C2}
	a_{N+1}\leq \alpha a_N+C. \end{equation}
	Recursively for all $
	N\geq  0$ we obtain 
	\begin{align*}
	a_{N+1}&\leq \alpha^{N+1}a_{0}+C(1+\alpha+\alpha^2+...+\alpha^{N})
	\\
	&=\alpha^{N+1}a_{0}+C\frac{1-\alpha^{N+1}}{1-\alpha}.
	\end{align*}
	{Thus the sequence $(a_n)$ is bounded}.\\
	(ii) Then, the sequence $\{v(x_n)\}$ is bounded and since 
	\begin{equation}\label{C4}
	\sup_n\|v(x_{n+1})-v(x_n)\|\leq 2\sup_n\|v(x_n)\| <+\infty,
	\end{equation} 
	the sequence $\{\delta_n\}$ is also bounded.
	Combining \eqref{C4}  with \eqref{C1} and noticing that $1-3\alpha>0,$ yields  
	$$\sum_{n=1}^{+\infty}\delta_n<+\infty,$$
	ensuring the result. 
\end{proof}  

We now proceed to the statement and proof of the main result of this section.

\begin{theorem}\label{t:warped}
{Suppose that Assumption~1 holds,  \( S:=\zer F \neq \varnothing \), and  the sequence}
\( (\alpha_n)_{n\in\mathbb{N}} \) is nondecreasing with \( \alpha_n \in [0,\alpha] \) for some
\( \alpha \in [0,\tfrac{1}{3}) \).
Assume further that
\[
\gamma := \inf_{n\in\mathbb{N}} \gamma_n > 0.
\]
Let \( (x_n)_{n\in\mathbb{N}} \) be the sequence generated by Algorithm~\textbf{(GIPPA)}, and define
\[
y_n := x_n + \alpha_n (x_n - x_{n-1}), \qquad n \ge 1.
\]
Then the following assertions hold:
\begin{enumerate}
\item\label{t:warped_a}
\( \|v(y_n) - v(x_{n+1})\| \to 0 \) as \( n \to \infty \).

\item\label{t:warped_b}
{If \( F \) has a strongly--weakly closed graph,  then every weak cluster point \( x^* \) of \( (x_n) \) is a solution of \eqref{main}}.
In addition, if  \( F^{-1} \) is \( R \)-continuous at \( 0 \), then
\( d(x_n,S) \to 0 \) as \( n \to \infty \).

\item\label{t:warped_weakcvg}
If \( F \) has a strongly--weakly closed graph and {\( v \) is bijective}, 
then \( (x_n)_{n\in\mathbb{N}} \) converges weakly to a solution of \eqref{main}.

\item\label{t:warped_strongcvg}
If Assumption~1\( '\) holds, then \( (x_n)_{n\in\mathbb{N}} \) converges strongly to the unique solution
\( x^* \) of \eqref{main}. Moreover,
\begin{enumerate}
\item\label{t:warped_strongcvg_v}
the sequence \( (v(x_n))_{n\in\mathbb{N}} \) converges linearly to \( v(x^*) \), provided that
\( \alpha \) is sufficiently small;

\item\label{t:warped_strongcvg_x}
if {\( v \) is bijective}, then
\( (x_n)_{n\in\mathbb{N}} \) converges linearly to \( x^* \).
\end{enumerate}
\end{enumerate}
\end{theorem}

\medskip 
\vskip 2mm
\begin{proof}

\medskip
\ref{t:warped_a} 
%
{
Define
\[
u_n := \frac{v(y_n)-v(x_{n+1})}{\gamma_n}.
\]
Since \(x_{n+1}=J_{\gamma_n F}^v(y_n)\), we have \(u_n\in F(x_{n+1})\).
Using \eqref{eq:vv*} and \eqref{8a}, we obtain
\begin{align*}
\gamma_n^2\|u_n\|^2
&= \|v(y_n)-v(x_{n+1})\|^2 \\
&\le \|v(y_n)-v(x^*)\|^2 - a_{n+1} \\
&= (a_n-a_{n+1}) + \alpha_n(a_n-a_{n-1})
   + \alpha_n(1+\alpha_n)\delta_n.
\end{align*}
Since \((\alpha_n)\) is nondecreasing and satisfies \(\alpha_n<1\), it follows that
\[
\gamma_n^2\|u_n\|^2
\le (a_n-a_{n+1}) + \alpha_n a_n - \alpha_{n-1} a_{n-1}
   + 2\delta_n.
\]
Summing from \(n=1\) to \(N\) and using \(\alpha_n\le \alpha\), we deduce
\[
\sum_{n=1}^{N}\gamma_n^2\|u_n\|^2
\le (a_1-a_{N+1}) + (\alpha a_N-\alpha_0 a_0)
   + 2\sum_{n=1}^{N}\delta_n.
\]
Letting \(N\to\infty\), and noting that \((a_n)\) is bounded and
\(\sum_{n\ge1}\delta_n<\infty\) (see Corollary~\ref{coro-disc}), we conclude that
\[
\sum_{n=0}^{\infty}\gamma_n^2\|u_n\|^2 < \infty.
\]
Consequently, since \(\gamma^2\|u_n\|^2 \le \gamma_n^2\|u_n\|^2\) and
\(\gamma_n^2\|u_n\|^2 \to 0\), it follows that \(\|u_n\|\to 0\) and hence
\[
\|v(y_n)-v(x_{n+1})\|\to 0.
\]
}
\medskip


\ref{t:warped_b}
{
Note that \(u_n\in F(x_{n+1})\).
Let \(x^*\) be a weak cluster point of the sequence \((x_n)\).
Since \(u_n\to 0\) and the graph of \(F\) is strongly--weakly closed, we obtain
\(0\in F(x^*)\), and hence \(x^*\in S\).
}
{
Moreover, if \(F^{-1}\) is \(R\)-continuous at \(0\), then for \(n\) sufficiently large,
\[
x_{n+1}\in F^{-1}(u_n)
\subset F^{-1}(0)+\rho(\|u_n\|)\mathbb{B}
= S+\rho(\|u_n\|)\mathbb{B},
\]
where \(\rho\) denotes the associated modulus of continuity.
Consequently,
\[
d(x_{n+1},S)\le \rho(\|u_n\|)\to 0,
\]
which completes the proof.
}

\ref{t:warped_weakcvg} 
{
From inequality~\eqref{estim1}, we obtain
\begin{align*}
a_{n+1}-a_n 
&\leq \alpha_n (a_n-a_{n-1})
+(\alpha_n-1)\|v(x_{n+1})-v(x_n)\|^2
+2\alpha_n\|v(x_n)-v(x_{n-1})\|^2.
\end{align*}
Taking positive parts and using the fact that \(\alpha_n\le \alpha<1\), it follows that
\[
[a_{n+1}-a_n]_+
\le \alpha [a_n-a_{n-1}]_+
+2\alpha\|v(x_n)-v(x_{n-1})\|^2.
\]
By Corollary~\ref{coro-disc}, we have
\[
\sum_{n=1}^{\infty}\|v(x_n)-v(x_{n-1})\|^2 < \infty.
\]
Applying Lemma~\ref{lem-sum} with
\[
b_n := [a_n-a_{n-1}]_+, 
\qquad 
w_n := 2\alpha\|v(x_n)-v(x_{n-1})\|^2,
\]
yields
\[
\sum_{n=1}^{\infty}[a_n-a_{n-1}]_+ < \infty.
\]
Since the sequence \((a_n)\) is nonnegative, an argument similar to that used in the proof of Lemma~\ref{lem-sum1} shows that \((a_n)\) converges. Consequently, the limit
\[
\lim_{n\to\infty}\|v(x_n)-v(x^*)\|
\]
exists. Therefore, condition~(i) of Lemma~\ref{l:Opial} is satisfied.
}
\medskip

%
%
\medskip

We now show that every weak cluster point of the sequence
\((v(x_n))_{n\in\mathbb{N}}\) belongs to \(v(\zer F)\).
Let \(\bar{w}\) be a weak cluster point of \((v(x_n))_{n\in\mathbb{N}}\).
Then there exists a subsequence \((v(x_{k_n}))_{n\in\mathbb{N}}\) such that
\(v(x_{k_n}) \rightharpoonup \bar{w}\).
Since \(v\) is linear and bijective, it follows that
\[
x_{k_n} \rightharpoonup \bar{x} := v^{-1}(\bar{w}).
\]
Moreover, since \(u_{k_n-1} \in F(x_{k_n})\) and the graph of \(F\) is
strong--weakly closed, we obtain \(0 \in F(\bar{x})\).
Consequently,
\[
\bar{w} = v(\bar{x}) \in v(\zer F).
\]

By Lemma~\ref{l:Opial}, the sequence \((v(x_n))_{n\in\mathbb{N}}\) converges
weakly to some \(\tilde{w} \in v(\zer F)\).
Therefore, \((x_n)\) converges weakly to
\(\tilde{x} := v^{-1}(\tilde{w}) \in \zer F\), since \(v^{-1}\) is weakly
continuous.

\medskip
\ref{t:warped_strongcvg}: 
{Since the pair \((F,v)\) is strongly monotone, relation~\eqref{strong}
implies that \eqref{main} admits a unique solution \(x^*\).
Furthermore, the pair \((\gamma_n F, v)\) is also strongly monotone, and
the operator \(\gamma_n F + v\) is injective for all \(n\).
Hence, Assumption~\(1'\) is stronger than Assumption~1.}

{Using assertions~\ref{t:warped_a} and~\ref{t:warped_weakcvg}, we deduce that the
sequence \((\|v(x_n)-v(x^*)\|)\) is convergent and that
\(\|v(y_n)-v(x_{n+1})\|\to 0\).
By \cite[Proposition~3.1(ii)]{LDT}, we have \(x^* \in \Fix J_{\gamma_n F}^v\),
and hence \((x^*,x^*) \in \gra J_{\gamma_n F}^v\).
Since \((y_n,x_{n+1}) \in \gra J_{\gamma_n F}^v\),
\cite[Proposition~3.3(ii)]{LDT} yields
$$
\beta \gamma \|x_{n+1}-x^*\|^2
+ \|v(x_{n+1})-v(x^*)\|^2
\le
\|v(y_n)-v(x^*)\|\,\|v(x_{n+1})-v(x^*)\|,
$$
where \(\gamma := \inf_{n\in\mathbb{N}} \gamma_n\).
}
{
Therefore, we obtain
\begin{align*}
\beta\gamma \|x_{n+1}-x^*\|^2
\le \|v(x_{n+1})-v(x^*)\|\,\|v(y_n)-v(x_{n+1})\|
\;\xrightarrow[n\to\infty]{}\;0.
\end{align*}
Consequently, the sequence \((x_n)\) converges strongly to \(x^*\).
}
\medskip

\noindent\ref{t:warped_strongcvg_v}.  
{
Since \(v\) is linear, it is \(L\)-Lipschitz continuous for some \(L>0\).
By \cite[Proposition~3.3(ii)]{LDT}, we have
\[
\|v(x_{n+1})-v(x^*)\|
\le \frac{1}{1+\beta\gamma_n L^{-2}}\,\|v(y_n)-v(x^*)\|
\le \kappa\,\|v(y_n)-v(x^*)\|,
\]
where
\[
\kappa := \frac{1}{1+\beta\gamma L^{-2}} < 1,
\qquad
\gamma := \inf_{n\in\mathbb{N}}\gamma_n.
\]
}
{
Using \eqref{8a} and the fact that \(\alpha_n \le 1\), we obtain
\begin{align*}
\|v(y_n)-v(x^*)\|^2
&\le \|v(x_n)-v(x^*)\|^2
+ \alpha_n\bigl(\|v(x_n)-v(x^*)\|^2-\|v(x_{n-1})-v(x^*)\|^2\bigr) \\
&\quad + \alpha_n(1+\alpha_n)\|v(x_n)-v(x_{n-1})\|^2 \\
&\le (1+5\alpha_n)\|v(x_n)-v(x^*)\|^2
+ 3\alpha_n\|v(x_{n-1})-v(x^*)\|^2.
\end{align*}
It follows that
\[
\|v(x_{n+1})-v(x^*)\|^2
\le \kappa^2(1+5\alpha)\|v(x_n)-v(x^*)\|^2
+ 3\kappa^2\alpha\|v(x_{n-1})-v(x^*)\|^2.
\]
}
{
Consequently,
\begin{align*}
\|v(x_{n+1})-v(x^*)\|^2
+ \tilde{\alpha}\|v(x_n)-v(x^*)\|^2
\le \rho^2\bigl(\|v(x_n)-v(x^*)\|^2
+ \tilde{\alpha}\|v(x_{n-1})-v(x^*)\|^2\bigr),
\end{align*}
where the parameters \(\alpha,\tilde{\alpha}>0\) are chosen such that
\[
\rho := \sqrt{\kappa^2(1+5\alpha)+\tilde{\alpha}} < 1,
\qquad
\frac{3\kappa^2\alpha}{\kappa^2(1+5\alpha)+\tilde{\alpha}} \le \tilde{\alpha}.
\]
For instance, one may take \(\tilde{\alpha}=(1-\kappa^2)/2\) and choose
\(\alpha>0\) sufficiently small.
}

{By induction, we obtain
\begin{align*}
\|v(x_{n+1})-v(x^*)\|^2
+ \tilde{\alpha}\|v(x_n)-v(x^*)\|^2
\le \rho^{2n}\bigl(\|v(x_1)-v(x^*)\|^2
+ \tilde{\alpha}\|v(x_0)-v(x^*)\|^2\bigr),
\end{align*}
which implies
\[
\|v(x_{n+1})-v(x^*)\|
\le \rho^n
\sqrt{\|v(x_1)-v(x^*)\|^2
+ \tilde{\alpha}\|v(x_0)-v(x^*)\|^2}.
\]
This establishes the claimed linear convergence and completes the proof.
}

\medskip
\ref{t:warped_strongcvg_x}
If $v$ is bijective and linear then $v^{-1}$ is  $\ell$-Lipschitz continuous for some $\ell>0$. Then
\begin{align*}
\|x_n - x^*\| 
&\leq \ell \|v(x_{n+1}) - v(x^*)\| \\
&\leq \ell \rho^n \sqrt{\|v(x_1)-v(x^*)\|^2+\tilde{\alpha}\|v(x_{0})-v(x^*)\|^2}.
\end{align*}
Thus, $(x_n)$ converges linearly to $x^*$, which completes the proof.  

\end{proof}

\begin{remark}
{The properties stated in \ref{t:warped_a} and \ref{t:warped_b} are reminiscent of those of the \(\mathbf{(DCA)}\) (Difference of Convex Algorithm) in nonconvex programming; see \cite{Pham}. While DCA is fundamentally based on a \emph{difference} of convex functions, the algorithms proposed here rely instead on a \emph{product} structure, namely the monotonicity of pairs of operators. }

{Moreover, when \(\alpha_n=0\), that is, when the \(\mathbf{(GIPPA)}\) reduces to the \(\mathbf{(GPPA)}\), these properties continue to hold without requiring the mapping \(v\) to be linear.
}
\end{remark}

\section{Numerical Examples}

In this section, we present simple illustrative examples, implemented in Scilab~5.5.2, that highlight the theoretical results of Theorem~\ref{t:warped} and serve as a benchmark for numerical experiments. These experiments compare the performance of the inertial scheme $(\mathbf{GIPPA})$ with the classical generalized proximal point algorithm $(\mathbf{GPPA})$ introduced in \cite{Arag,LDT}.

\begin{example}
	
	Consider the operator \(F:\mathbb{R}^3 \to \mathbb{R}^3\) defined by
	\[
	F(x) = A x - b,
	\]
	where
	$$
	A =\begin{bmatrix} 1 & 2 & 3\\ 4 & 5 & 6 \\ 7 & 8 & 9\end{bmatrix},  \quad
	b = \begin{bmatrix} 14 \\ 32 \\ 50\end{bmatrix}.
	$$
	Then $A$ is non-monotone and non-invertible. The solution set is non-empty since it has a solution $x^*=(1,2,3)$.
	We choose for example
	$$
	v=v_1=\begin{bmatrix} 2 & 2 & 3\\ 4 & 5 & 6 \\ 7 & 8 & 9\end{bmatrix},
	$$
	then $v_1$ is invertible and $(F,v_1)$ is monotone.

	Starting from $x_0=(-0.5,-0.5,-0.5)$ and $x_1=(0.7,0.7,0.7)$, we evaluate the iterate error $\|x_n-x^*\|_2$ to compare the performance of $(\mathbf{GPPA})$ and $(\mathbf{GIPPA})$. 
	
	
	\begin{figure}[h]
		\subfloat[For $\gamma_n=0.1 + \frac{0.3}{n+10}$ and various $\alpha_n$]{\includegraphics[scale=0.4]{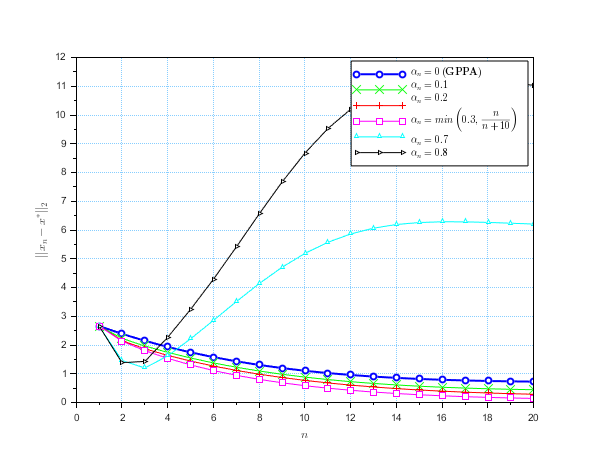}}
		\subfloat[For $\alpha_n=\min\{0.3, \frac{n}{n+10}\}$ and various $\gamma_n$]{\includegraphics[scale=0.4]{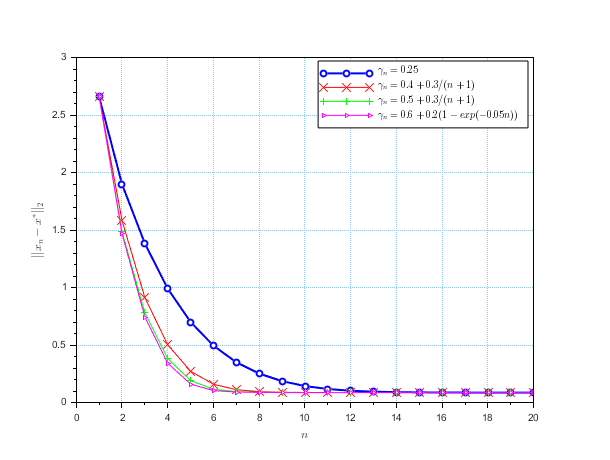}}
		\caption{The rate of convergence of $\|x_n-x^*\|_2$ {when $v=v_1$}.}
		\label{error0}
	\end{figure} 
	
	Figure \ref{error0}(a) illustrates the convergence rate of $\|x_n-x^*\|_2$ for $\gamma_n=0.1 + \frac{0.3}{n+10}$ and different values of the inertial parameter $\alpha_n$.  
	The results show that $\|x_n-x^*\|_2$ decreases rapidly when $\alpha_n \in [0,\tfrac{1}{3})$, where $(\mathbf{GIPPA})$ significantly outperforms $(\mathbf{GPPA}) (\alpha_n=0)$.  
	However, when $\alpha_n$ exceeds $\tfrac{1}{3}$, the convergence deteriorates, confirming the theoretical restriction $\alpha_n<\tfrac{1}{3}$. Meanwhile, Figure \ref{error0}(b) presents the convergence behavior of $\|x_n - x^*\|_2$ for various choices of $\gamma_n$, where $\alpha_n = \min\{0.3, \frac{n}{n+10}\}$. It is observed that the convergence improves as the lower bound $\gamma:=\inf_{n\in \N} \gamma_n$ increases.\\
	

{
Next, we choose
\[
v = v_2 =
\begin{bmatrix}
1 & 0 & 0\\
0 & 0 & 0\\
0 & 0 & 0
\end{bmatrix}.
\]
Then the pair \((F,v_2)\) is monotone, and the corresponding warped resolvent \(J_A^{v_2}\) is well-defined. 
We note that this choice of \(v_2\) is computationally less expensive than the choice of \(v_1\).
In this case, we obtain the iterative scheme illustrated in Figure~2.
}

\begin{figure}[h]
	\centering {\includegraphics[scale=0.4]{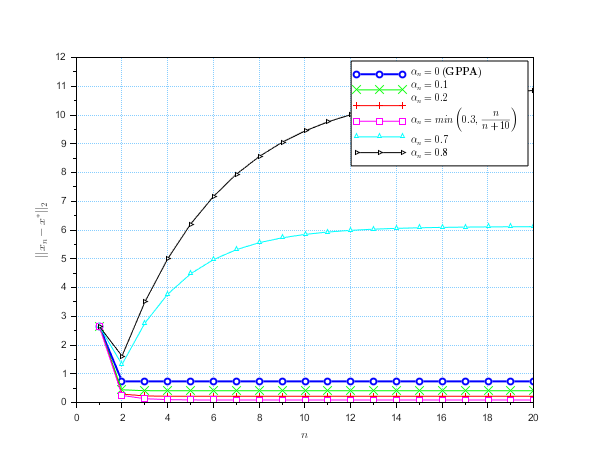}}
	\caption{\centering {The iterate error $\|x_n-x^*\|_2$  when $v=v_2$}}
\end{figure}	
\end{example}

\begin{example}
	Consider the operator \({f}:\mathbb{R}^3 \to \mathbb{R}^3\) defined by
	\[
	{f}(x) = {A} x +{g}(x),
	\]
	where
	
	$$
	{A} =
	\begin{bmatrix}-1 &0&0\\0 & 5 & 0 \\ 0 & 0 & 9\end{bmatrix},  \\
	\quad	{g}(x) = 
	\begin{bmatrix}2\sin(\vert x_3\vert+x_2)\\\cos (\vert x_1\vert-x_2)\\ 2\cos  x_2 -3\sin\vert x_3\vert\end{bmatrix}.
	$$
	Then $A$ and $g$ are non-monotone but $(f,A)$ is monotone.  In this case, one can choose the kernel $v=A$. 
	
	Starting from initial values $x_0=(2,-2,1)$ and $x_1=(1.5,-1.5,0.5)$, the numerical results in Figure \ref{error} illustrate the convergence behavior of $\|x_n-x^*\|_2$ for an approximate solution $x^*\approx(-0.06,-0.195,-0.164)$.
	The plot in Figure \ref{error}(a) illustrate the rate of convergence of $\|x_n-x^*\|_2$ for various choices of $\alpha_n$ and fixed $\gamma_n=0.5$, while Figure \ref{error}(b) displays $\|x_n-x^*\|_2$ for $\alpha_n=\min\{0.3, \frac{n}{n+10}\}$ and different choices of $\gamma_n$. 
	
	\begin{figure}[h]
		\subfloat[For $\gamma_n=0.5$ and various $\alpha_n$]{ \includegraphics[scale=0.4]{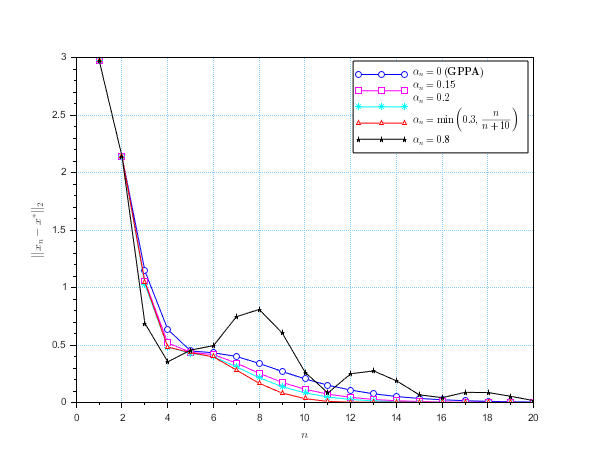}}
		\subfloat[For $\alpha_n=\min\{0.3, \frac{n}{n+10}\}$ and various $\gamma_n$]{\includegraphics[scale=0.4]{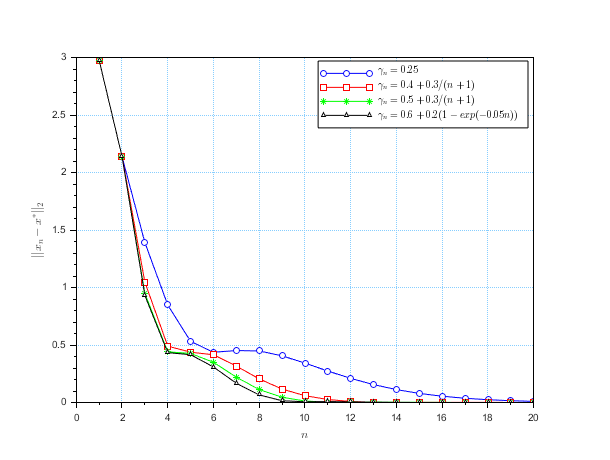}}
		\caption{The rate of convergence of $\|x_n-x^*\|_2$.}
		\label{error}
	\end{figure} 


	As in the previous example, Figure \ref{error}(a) clearly shows that  $(\mathbf{GIPPA})$ outperforms $(\mathbf{GPPA})$, while Figure \ref{error}(b) indicates that the convergence improves with larger values of $\gamma:=\inf_{n\in \N} \gamma_n$.
	
\end{example}

\section{Conclusion and Perspectives}

In this paper, we have addressed the question of how to construct the associated mappings so that the resulting pairs are monotone and proposed an inertial proximal-type algorithm based on warped resolvents for solving non-monotone inclusions, leveraging the monotonicity of pairs of operators.  Under mild assumptions, we have established weak, strong, and linear convergence of the method, without imposing restrictive conditions on the trajectories. Numerical examples further demonstrate the algorithm's effectiveness in comparison with the method in \cite{LDT}.

Looking ahead, we believe that the framework of monotone operator pairs opens up a promising and largely unexplored research direction. In particular, developing algorithms for problems involving the sum of two operators--where one or both may fail to be monotone--could significantly broaden the scope of resolvent-based techniques. This includes designing new splitting schemes and operator-splitting strategies that strategically exploit pair monotonicity to secure robust convergence guarantees. Continued exploration of these ideas has the potential to reshape algorithmic design for non-monotone problems, extend the reach of monotone operator theory, and stimulate further advances in large-scale optimization and variational analysis.

%


\end{document}